\documentclass[11pt]{amsart}

\theoremstyle{plain}
\newtheorem{thm}{Theorem}[section]
\newtheorem{theorem}[thm]{Theorem}

\newtheorem{lemma}[thm]{Lemma}

\newtheorem{proposition}[thm]{Proposition}
\theoremstyle{definition}
\newtheorem{remark}[thm]{Remark}

\newtheorem{definition}[thm]{Definition}

\newtheorem{conjecture}[thm]{Conjecture}

\numberwithin{equation}{section}

\newcommand{\sK}{{\mathcal K}}

\newcommand{\sO}{{\mathcal O}}

\newcommand{\sR}{{\mathcal R}}
\newcommand{\sS}{{\mathcal S}}

\newcommand{\sU}{{\mathcal U}}

\newcommand{\sX}{{\mathcal X}}

\newcommand{\BP}{{\mathbb P}}

\newcommand{\fg}{{\mathfrak g}}

\newcommand{\fh}{{\mathfrak h}}

\def\Hom{\mathop{\rm Hom}\nolimits}

\title[the formal principle]{An application of Cartan's equivalence method to   Hirschowitz's conjecture on the formal principle}
\author[Jun-Muk Hwang]{Jun-Muk Hwang} 
\address{Korea Institute for Advanced Study, Hoegiro 85, Seoul 02455, Korea}
\email{jmhwang@kias.re.kr}
\thanks{The author is supported
by National Researcher Program 2010-0020413 of NRF}

\begin{document}

\begin{abstract} A conjecture of Hirschowitz's  predicts that  a globally generated vector bundle $W$ on a compact complex manifold $A$ satisfies the formal principle, i.e., the formal neighborhood of its zero section determines the germ of  neighborhoods in the underlying complex manifold  of the vector bundle $W$. By applying Cartan's equivalence method to a suitable differential system on the universal family of the Douady space of the complex manifold, we prove that this conjecture is true if $A$ is a Fano manifold, or if the global sections of $W$  separate points of $A$. Our method  shows more generally that for any unobstructed compact submanifold $A$ in a complex manifold, if the normal bundle is globally generated and its sections separate points of $A$, then a sufficiently general deformation of $A$ satisfies the formal principle. In particular, a sufficiently general smooth free rational curve on a complex manifold satisfies the formal principle.
\end{abstract}

\maketitle



\section{Introduction}
A fundamental problem in complex geometry is to understand the germ, or the neighborhood structure,  of a (compact) complex submanifold in a complex manifold (see \cite{Gi66} for an introduction to various questions on this topic and
\cite{ABT}  for more recent developments).
A natural approach is to study first the isomorphism types of the finite-order neighborhoods of the submanifold,
 which is usually a cohomological question of geometric or algebraic nature. Once all the finite-order neighborhoods are understood, i.e., the isomorphism type of the formal neighborhood of the submanifold is determined, then one faces the question of the convergence of the formal isomorphism or the existence of  biholomorphic isomorphisms approximating the formal isomorphism, which is usually a question of analytic nature.  Our main interest, the formal principle, is a version of the latter question.
 We use the following definition.

 \begin{definition}\label{d.GP} For a compact complex submanifold $A$ in a complex manifold $X$,
 \begin{itemize}\item[(i)] $(A/X)_{\ell}$ for  a nonnegative integer $\ell$ denotes
its $\ell$-th order neighborhood (i.e. the analytic space defined
  by the $(\ell+1)$-th power of the ideal of $A \subset X$); \item[(ii)] $(A/X)_{\infty}$ denotes the formal neighborhood of $A$ in $X$; and \item[(iii)] $(A/X)_{\sO}$ denotes the germ of (Euclidean) neighborhoods of $A$ in $X$. \end{itemize}
  We say that $A \subset X$
 {\em satisfies the formal principle}, or equivalently, {\em the formal principle holds} for $A \subset X$,
    if given \begin{itemize}
        \item[(1)] a  compact submanifold $\tilde{A}$ in a complex manifold $\tilde{X}$;
    \item[(2)]
      a formal isomorphism $\psi: (A/X)_{\infty} \to (\tilde{A}/\tilde{X})_{\infty}$ between the formal neighborhoods; and \item[(3)] a positive integer $\ell$,  \end{itemize} we can find a biholomorphism
$\Psi: (A/X)_{\sO} \to (\tilde{A}/\tilde{X})_{\sO}$  such that $\Psi|_{(A/X)_{\ell}} = \psi|_{(A/X)_{\ell}}$.
\end{definition}

  As many authors have been interested in  comparing the germ $A\subset X$ with the germ of  the zero section of the normal bundle $N_{A/X}$ (see \cite{ABT} and references therein),  it is convenient to introduce the following terminology.

\begin{definition}\label{d.bundle}
Let $W$ be a vector bundle on a compact complex manifold $A$. We say that $W$ {\em satisfies the
formal principle}, if the zero section $0_A \subset W$, regarded as a submanifold of the complex manifold $X$ underlying the vector bundle $W$, satisfies the formal principle in the sense of Definition \ref{d.GP}.
    \end{definition}

 Not every compact complex submanifold in a complex  manifold satisfies the formal principle. In fact, Arnold (\cite{Ar}) discovered  a line bundle of degree 0 on an elliptic curve that does not satisfy the formal principle.   On the other hand, the formal principle holds in many cases of $A \subset X$, if the normal bundle $N_{A/X}$ satisfies certain positivity or negativity conditions (see the surveys in \cite{Ks86} and Section VII.4 of \cite{GP}). In \cite{Hi}, Hirschowitz explored the possibility of replacing complex-analytic or differential-geometric positivity conditions on the normal bundle $N_{A/X}$ by  more geometric conditions in terms of the deformations of the submanifold $A$ in $X$ and proposed a very interesting conjecture.  To avoid technicalities, let us assume that the submanifold $A \subset X$ is unobstructed, i.e., all infinitesimal deformations $H^0(A, N_{A/X})$ can be realized as actual deformations of $A$ in $X$ (this is equivalent to the smoothness of the Douady space of $X$ at the point corresponding to $A$). Then we can state Hirschowitz's conjecture in the introduction of \cite{Hi} in the following way.

\begin{conjecture}\label{c.Hirs}
Let $A  \subset X$ be an unobstructed compact submanifold of a complex manifold. Assume that the normal bundle  $N_{A/X}$ is globally generated, i.e., the sequence
$$ 0 \to H^0(A, N_{A/X} \otimes {\bf m}_x) \to H^0(A, N_{A/X}) \to N_{A/X,x} \to 0,$$
 where ${\bf m}_x$ is the maximal ideal at $x \in A$, is exact at every $x \in A$.  Then $A \subset X$ satisfies the formal principle. \end{conjecture}

Note that the zero section $0_A$ in Definition \ref{d.bundle} is always unobstructed. So  Conjecture \ref{c.Hirs} predicts the following.

\begin{conjecture}\label{c.bundle}
A globally generated vector bundle  on a compact complex manifold satisfies the formal principle. \end{conjecture}

The assumption of global generation in Conjectures \ref{c.Hirs} and \ref{c.bundle} is a geometric version of the semi-positivity of the normal bundle. In the two extreme cases of the semi-positivity, i.e., either when the normal bundle is trivial or when the normal bundle is positive, the conjecture was solved. Indeed, Theorem 2.2 of \cite{Ks86}, attributed to Hirschowitz, proves Conjecture \ref{c.Hirs} when $N_{A/X}$ is trivial. When the  normal bundle $N_{A/X}$ is ample in Conjecture \ref{c.Hirs}, Commichau and Grauert's result in \cite{CG} settles it (see the remark at the end of Section 4 in \cite{Ks86}). Also,
Hirschowitz \cite{Hi} had obtained some results on Conjecture \ref{c.Hirs}
under the additional assumption that $A$ has sufficiently many deformations in $X$ that have nonempty
intersections with $A$, which is a version of the positivity of the normal bundle.  These results cover  other works on the formal principle for submanifolds with positive normal bundle,  such as Theorem II of \cite{Gi66}. Their works were generalized to some singular varieties in \cite{Ks88} and \cite{St}. Since then, however,  there has been
 little progress on this problem.

   A difficulty in attacking the semi-positive situation of Conjecture \ref{c.Hirs} by the methods used in the works cited above  is due to a fundamental difference of the approaches in the trivial normal bundle case and the positive normal bundle case.   Among others, the methods used for the positive normal bundle case (both  \cite{CG} and \cite{Hi})  proved the {\em convergence} of the given formal isomorphism $\psi$ in Definition \ref{d.GP}, while  the convergence cannot be expected in the trivial normal bundle case.
It is hard to see how to combine these two  different approaches.

 \medskip
 In this paper, we employ \'E. Cartan's equivalence method for geometric structures, to obtain some new results on Conjecture \ref{c.Hirs} and  Conjecture \ref{c.bundle}.  Our main result is formulated in terms of Douady spaces. Recall that  for each complex space $X$, we have its Douady space denoted by ${\rm Douady}(X)$, a complex space parametrizing compact complex subspaces of $X$, with the associated
 universal family morphisms
 $$ {\rm Douady}(X) \stackrel{\rho}{\leftarrow}  {\rm Univ}(X) \stackrel{\mu}{\rightarrow} X.$$
The assignment of ${\rm Douady}(X)$ to a complex space $X$ is a functor, which is a complex-analytic version of the Hilbert scheme functor  in algebraic geometry. We refer the reader to \cite{Dou} or the introductory survey in Section VIII.1 of \cite{GP} for a detailed presentation of Douady spaces.

When interpreted in terms of the Douady space of $X$, the assumptions of Conjecture \ref{c.Hirs} say that the Douady space ${\rm Douady}(X)$ is smooth at $[A] \in {\rm Douady}(X)$ and the  morphism $\mu$ is submersive along the fiber $\rho^{-1}([A])$, i.e., the differential ${\rm d}_y \mu : T_y {\rm Univ}(X) \to T_{\mu(y)} X$ is surjective at every $y \in \rho^{-1}([A])$.
   Our main result is the following weak version of Conjecture \ref{c.Hirs}.

 \begin{theorem}\label{t.main}
 Let $X$ be a complex manifold and let $\sK \subset {\rm Douady}(X)$ be a subset  of the Douady space of $X$ with the associated universal family morphisms
 $$ \sK \stackrel{\rho}{\leftarrow} \sU := \rho^{-1}(\sK) \subset {\rm Univ}(X) \stackrel{\mu}{\rightarrow} X$$ such that
 \begin{itemize}
\item[(i)] $\sK$ is a connected open subset in the smooth loci of ${\rm Douady}(X)$; \item[(ii)] $\rho|_{\sU}$ is a smooth morphism with connected fibers;
 \item[(iii)] $\mu$ is submersive at every point of $\sU$; and
 \item[(iv)] for the submanifold  $A \subset X$  corresponding to any point in $\sK$,
 the normal bundle $N_{A/X}$ satisfies for any $ x \neq y \in A$,
$$H^0(A, N_{A/X} \otimes {\bf m}_x) \neq H^0(A, N_{A/X} \otimes {\bf m}_y) $$ as subspaces of $H^0(A, N_{A/X})$. \end{itemize} Then there exists a nowhere-dense subset $\sS \subset \sK$ such that the submanifold  $A \subset X$ corresponding to any  point of $\sK \setminus \sS$ satisfies the  formal principle. \end{theorem}

Note that all the conditions (i)--(iv) of Theorem \ref{t.main} are open conditions on ${\rm Douady}(X)$.
In the setting of Conjecture \ref{c.Hirs}, an open neighborhood $\sK$ of the point of ${\rm Douady}(X)$ corresponding to $A$ satisfies (i), (ii) and (iii). Thus Theorem \ref{t.main} says that if  the sections of the normal bundle separate points of $A$ in the setting of Conjecture \ref{c.Hirs}, then {\em the formal principle holds for sufficiently general deformations} of $A$ in $X$.

The key idea of the proof of Theorem \ref{t.main} is as follows.
We introduce    a natural system of differential equations at a general point of $\sU$. If we can  solve this system of differential equations, a special local holomorphic solution $\Psi$ in Definition \ref{d.GP} can be obtained  in a neighborhood of a point of $A$ and then we can analytically continue  it along $A$ to obtain a global solution. The additional condition (iv) in Theorem \ref{t.main} is imposed to prevent the multi-valuedness of the analytic continuation. Thus the problem is reduced to solving a locally defined  system of differential equations.
 Our system of differential equations describes an equivalence problem of certain geometric structures and it is solved by applying  a result of Morimoto's  in \cite{Mo83}.  Morimoto's result, which is a  rigorous version of Cartan's equivalence method (\cite{Ca}),  says  that the formal equivalence of geometric structures implies their biholomorphic equivalence,  at any point outside a nowhere-dense subset. The existence of the biholomorphic equivalence follows  eventually from Cartan-K\"ahler theorem with estimates (e.g. the version given in Appendix of \cite{Ma72} or Theorem IX.2.2 of \cite{BCG}). Thus our proof of Theorem \ref{t.main} is essentially a series of geometric arguments, reducing it to   its main analytical ingredient, Cartan-K\"ahler theorem.

Theorem \ref{t.main} has several applications.
 In Conjecture \ref{c.bundle}, the germ of any holomorphic section of the bundle $W \to A$ is biholomorphic to the germ of the zero section $0_A$ in $W$. Thus Theorem \ref{t.main} implies the following weaker version of Conjecture \ref{c.bundle}.

 \begin{theorem}\label{t.bundle}
 Let $W$ be a globally generated vector bundle on a compact complex manifold $A$ such that $$H^0(A, W \otimes {\bf m}_x) \neq H^0(A, W \otimes {\bf m}_y) \mbox{ for any } x \neq y \in A.$$
 Then $W$ satisfies the formal principle. \end{theorem}

Recall that a compact complex manifold $A$ is Fano if the anti-canonical line bundle $K^{-1}_A = \det TA$ is ample (or positive). We have the following refined version of Theorem \ref{t.bundle}, which proves Conjecture \ref{c.bundle} for Fano manifolds.

\begin{theorem}\label{t.Fano}
A globally generated vector bundle on a  Fano manifold satisfies the formal principle. \end{theorem}

Theorems \ref{t.main}, \ref{t.bundle} and  \ref{t.Fano} are new, even when the submanifold is  the Riemann sphere $\BP^1$.  The situation of $A = \BP^1$ in Theorem \ref{t.main} is actually the original motivation of this work. It implies the following.

\begin{theorem}\label{t.free}
Let $A \cong \BP^1 \subset X$ be a smooth rational curve whose normal bundle is globally generated
(such a rational curve is called a smooth free rational curve: see Section II.3 of \cite{Ko}). Let $\sK \subset {\rm Douady}(A)$ be a neighborhood of the point corresponding to $A$. Then there exists a nowhere-dense subset $\sS \subset \sK$ such that any member of $\sK \setminus \sS$  satisfies the formal principle. \end{theorem}

When combined with the Cartan-Fubini type extension theorem in \cite{HM01}, it gives  the following.

\begin{theorem}\label{t.CF}
Let $X, \tilde{X}$ be Fano manifolds of Picard number 1. Let $\sK$ (resp. $ \tilde{\sK}$) be an irreducible component of the space (e.g. ${\rm RatCurves}(X)$ in \cite{Ko}) of rational curves on $X$ (resp.  $\tilde{X}$)  such that the subscheme
$\sK_x \subset \sK$ (resp. $\tilde{\sK}_{\tilde{x}}$) consisting of members through a general point
$x \in X$ (resp. $\tilde{x} \in \tilde{X}$) is nonempty, projective and irreducible.   Then there exists a nowhere-dense subset $\sS \subset \sK$ such that
for any member  $A \subset X$ of $\sK \setminus \sS$, if there exists a member $\tilde{A}$ of $\tilde{\sK}$
equipped with a formal isomorphism
 $$\varphi: (\Gamma_A/ (\BP^1 \times X))_{\infty} \to (\Gamma_{\tilde{A}}/ (\BP^1 \times \tilde{X}))_{\infty}$$ where
 $\Gamma_A \subset \BP^1 \times X$ (resp. $\Gamma_{\tilde{A}} \subset \BP^1 \times \tilde{X}$) is the graph  of the normalization of $A$ (resp. $\tilde{A}$),  then $\varphi$ can be extended to a biholomorphic map from $X$ to $\tilde{X}$.
  \end{theorem}

 The original statement (e.g. Theorem \ref{t.CFO} below)  of Cartan-Fubini type extension theorem in \cite{HM01} involves transcendental conditions, in terms of Euclidean neighborhoods of rational curves. Because such transcendental conditions are not easy  to check effectively, the applicability of the Cartan-Fubini type extension theorem has been  limited. To remedy this,  it is essential to replace the transcendental conditions by algebraic conditions.    Theorem \ref{t.CF} is the
first step in this direction.

  When the curves $A$ and $\tilde{A}$ are singular in Theorem \ref{t.CF}, one may wonder whether it is more natural to formulate the condition in terms of  the formal neighborhoods of $A$ and $\tilde{A}$, not
  those of $\Gamma_{A}$ and $\Gamma_{\tilde{A}}$. Such a formulation might be possible using the
   notion of the formal principle for singular subvarieties, but it would be less useful.  In the study of families of rational curves covering projective varieties,  conditions in terms of the graph of the normalization are more effectively applicable than those in terms of singular curves.

The rest of the paper consists of two sections.
In Section 2, we give a streamlined review of  Morimoto's work on Cartan's equivalence method, with some modifications needed for our purpose. The proofs of Theorems \ref{t.main} -- \ref{t.CF} are given in Section 3.

Finally, let us mention that the novelty of our argument  lies in the observation that one can use Cartan's method to obtain results like Theorem \ref{t.main} by {\em viewing families of submanifolds on a complex manifold as a geometric structure in the sense of Cartan}. This viewpoint was already used in \cite{HM01} when the submanifolds are rational curves, and  can be useful also in the study of finite-order neighborhoods of complex submanifolds.

\bigskip
{\bf Acknowledgment}
I am very grateful to Takeo Ohsawa for drawing my attention to Conjecture \ref{c.Hirs}  for $A  = \BP^1$, which was the starting point of this work. He has also provided me with  valuable historical information on the subject of the formal principle.  I would like to thank Tohru Morimoto for helpful discussions on his work \cite{Mo83} and sending me a copy of \cite{TU}.

\section{Review of Morimoto's work on Cartan's equivalence method}

Roughly speaking, a geometric structure  on a complex manifold $M$ is some holomorphic data
imposed on the jet spaces of $M$. Locally, this is equivalent to a system of partial differential equations on $M$. The equivalence problem for geometric structures asks for methods to check whether two geometric structures of the same type are locally isomorphic or not.
\'Elie Cartan (\cite{Ca}) gave an outline of a general approach to solve the equivalence problem.
Rigorous realizations of Cartan's ideas have been presented by many authors in various settings. To our knowledge,   Morimoto's \cite{Mo83}, with a summary in \cite{Mo81}, is the first systematic account of the theory in full generality. Moreover, it states some of the results in the way most convenient for us. Since \cite{Mo83} is not widely known and is rather long,  we present a streamlined  review of part of it,  with some minor complements and
 modification necessary for our purpose.

 \medskip
 \subsection{}
Let us recall the terminology of Sections 2 and 3 in \cite{Mo83}.
 For a complex manifold $M$,
  there are naturally defined complex manifolds with holomorphic submersions
  $$ M \stackrel{\pi^0}{\leftarrow} \sR^0(M) \stackrel{\pi^1_0}{\leftarrow} \sR^1(M) \leftarrow \cdots \leftarrow \sR^k(M) \stackrel{\pi^{k+1}_k}{\leftarrow} \sR^{k+1}(M) \leftarrow$$ constructed inductively as follows. Fix a vector space $V$ with $\dim V = \dim M$. The submersion $\pi^0: \sR^0(M) \to M$ is the  frame bundle of $M$ with the fiber $\sR^0_x(M)$ at $x \in M$ equal to the set ${\rm Isom}(V, T_x M)$ of linear  isomorphisms from $V$ to the tangent space
  $T_{x} M$. For each $k \geq 0$, the manifold $\sR^{k+1}(M)$ is the set of all pairs
  $(z_k, H_k)$ consisting of a point  $ z_k $  of $\sR^k(M)$ and a subspace $H_k$ of $T_{z^k} \sR^k(M)$
   satisfying inductively
  $$\dim H_k = \dim H_{k-1}, \ {\rm d} \pi^k_{k-1} (H_k) = H_{k-1} \mbox{ and } {\rm d} \pi^k (H_k) = T_{\pi^k(z_k)} M,$$ where $z_k =(z_{k-1}, H_{k-1})$ by induction and $\pi^k: \sR^k(M) \to M$ is the composition $\pi^0 \circ \pi^1_0 \circ \cdots \circ \pi^k_{k-1}$. Then $\pi^{k+1}_k: \sR^{k+1}(M) \to \sR^k(M)$ is defined by   $\pi^{k+1}_k (z_{k+1}) = z_k$ if $z_{k+1} \in \sR^{k+1}(M)$ is given by $(z_k, H_k)$.

The submersion $\pi^k: \sR^k(M) \to M$ is a $G^k(V)$-principal bundle over $M$
    for a complex Lie group $G^k(V)$ with Lie algebra $\fg^k(V),$ described in Section 1.1 of \cite{Mo83}. The map $\pi^k_{k-1}$ is equivariant with respect to the actions of $G^k(V)$ and $G^{k-1}(V)$ related by a natural surjective group homomorphism $$\varepsilon^k_{k-1}: G^k(V) \to G^{k-1}(V).$$ The corresponding Lie algebra homomorphism is  denoted by the same symbol
    $$\varepsilon^k_{k-1} : \fg^k(V) \to \fg^{k-1}(V)$$ by abuse of notation.
The manifold $\sR^k(M)$ has a natural 1-form $\Theta^{k-1}$ with values in the vector space $V^{k-1} := V + \fg^{k-1}(V)$, called the {\em fundamental form}, which generalizes the soldering form on the frame bundle $\sR^0(M)$ (see p.307 of \cite{Mo83}).
The {\em structure  function} $C^{k-2}$  is a holomorphic function on $\sR^k(M)$ with values in $V^{k-2} \otimes \wedge^2 V^*$,  whose value $C^{k-2}(z)(u,v)  \in V^{k-2}$
for $u, v \in V$ and $z \in \sR^k(M)$ is given by  (see page 311 of \cite{Mo83})
 $$C^{k-2}(z)(u,v)\ := \ \langle (\pi^k_{k-1})^* {\rm d} \Theta^{k-2}, u^{\sharp} \wedge v^{\sharp} \rangle,$$ where $u^{\sharp}, v^{\sharp} \in T_{z} \sR^k(M)$ are any vectors satisfying $$\Theta^{k-1}(u^{\sharp}) =u \mbox{ and } \Theta^{k-1}(v^{\sharp}) = v.$$

A biholomorphic map $\Phi: M \to \tilde{M}$ between two complex manifolds induces, in a canonical way, a biholomorphic map $\Phi^{(k)}: \sR^k(M) \to \sR^k(\tilde{M})$ satisfying
$\Phi^{(k)*} \tilde{\Theta}^{k-1} = \Theta^{k-1}$ where $\tilde{\Theta}^{k-1}$ denotes the fundamental form on $\sR^{(k)}(\tilde{M})$. Similarly, for $x \in M$ and $\tilde{x} \in \tilde{M}$, a formal isomorphism $\varphi: (x/M)_{\infty} \to (\tilde{x}/\tilde{M})_{\infty}$  induces a formal isomorphism $$\varphi^{(k)}: ((\pi^k)^{-1}(x)/ \sR^{k} (M))_{\infty} \to ((\tilde{\pi}^k)^{-1}(\tilde{x})/
  \sR^{k}(\tilde{M}))_{\infty}.$$

\subsection{}

We recall the terminology of Section 4 of \cite{Mo83}.
A vector bundle $\varepsilon_B: \fg \to B$ on a complex manifold $B$ equipped with holomorphically varying  Lie algebra structures on its fibers $\{ \fg(b), b \in B\},$ is called a {\em Lie $B$-algebra}.  It defines a {\em Lie $B$-group germ}, i.e., a submersion  $\varepsilon_B: G \to B$ (denoted by the same symbol $\varepsilon_B$ by abuse of notation) of complex manifolds with a distinguished section $e: B \to G$ such that the fiber $G(b)$ of $G\to B$ at $b \in B$ is the Lie group germ of the  Lie algebra $\fg(b)$ with the identity $e(b)$.   A complex Lie group (resp. Lie algebra) is regarded as a Lie $B$-group germ (resp. a Lie $B$-algebra) for an isolated point $B$.
We say that a Lie $B$-group germ $\varepsilon_B: G \to B$ is a Lie {\em $B$-subgroup germ} of a Lie $B'$-group germ $\varepsilon_{B'}: G' \to B'$, if there exits  a holomorphic map $h: B \to B'$ and an immersion $ \iota: G \to G' \times_{B'} B $ such that $\varepsilon_{B'} \circ \iota = h \circ \varepsilon_B$ and $\iota|_{G(b)} : G(b) \to G'(h(b))$ is an injective homomorphism of Lie group germs for any $b \in B$. In this case, we say that $\fg \to B$ is a Lie $B$-{\em subalgebra} of a Lie $B'$-algebra $\fg' \to B'$.

The right-action of a Lie $B$-group germ $\varepsilon_B: G \to B$ on a manifold $P$ with a submersion
$P\to B$ is defined as a holomorphic map $\alpha: U \to P$ defined on a  neighborhood $U \subset P \times_B G$ of $P \times_B e(B) \subset P \times_B G$ satisfying the usual group action property $\alpha(z, e(b)) = z$ for any $z \in P, b \in B$ and
$$ \alpha (z, g_1 \cdot g_2) = \alpha (\alpha(z, g_1), g_2),  $$ whenever both sides make sense.

Let $\pi: P \to M$ be a submersion of complex manifolds. Suppose there exist a Lie $B$-group germ $\varepsilon_B: G \to B$ and a submersion
$\pi_B: M \to B$ with a right action of $G\to B$ on $\pi_B \circ \pi:  P \to B$ given
by $\alpha: U \to P$ for a neighborhood $U \subset P \times_B G$ of $P \times_B e(B)$. Denote by $P_x$ the fiber of $\pi: P \to M$ at $ x \in M$.
We call such data $P (M, B, G)$  a {\em principal $B$-bundle germ with
the structure group germ} $\varepsilon_B: G \to B,$
if  for any point $z \in P_x$  and  $ b = \pi_B(x) \in B$,  the orbit map of the action of the group germ  $e(b) \in G(b)$
 gives a  biholomorphism $(e(b)/ G(b))_{\sO} \cong (z/P_x)_{\sO}.$
   We have the notion of a principal $B$-subbundle germ in a similar way.

Let $\varepsilon^k_B: G^k \to B$ be a Lie $B$-subgroup germ of the Lie group $G^k(V)$
 for a nonnegative integer $k$. We say that it is a {\em regular} $B$-subgroup germ of $G^k(V)$, if for each $0 \leq i \leq k$, the image $G^i = \varepsilon^k_i(G^k)$ under the composition
  $$\varepsilon^k_i := \varepsilon^{i+1}_i \circ \cdots \circ \varepsilon^k_{k-1}: G^k(V) \to G^i(V)$$ is a Lie $B$-subgroup germ of $G^i(V)$. Then we have the Lie $B$-subalgebra $\fg^i \to B$ of the Lie algebra $\fg^i(V)$.

For a regular $B$-subgroup germ $G^k \to B$ of $G^k(V),$    a principal $B$-subbundle germ
$P^k(M, B, G^k)$ of $\sR^k(M)$ is called a {\em Cartan bundle} of order $k+1$ on $M$. The regularity of $G^k \to B$ implies that the image $\pi^k_i (P^k(M, B, G^k))$ defines a Cartan bundle $P^i(M, B, G^i)$ of order $i+1$ on $M$ for each $0 \leq i \leq k$.

 For two Cartan bundles  $P^k(M, B, G)$ on $M$ and $\tilde{P}^k(\tilde{M}, \tilde{B}, \tilde{G})$ on  $\tilde{M}$, a biholomorphism (resp. formal isomorphism) $$\Phi: M \to \tilde{M} \ ( \mbox{resp. }
 \varphi: (x/M)_{\infty} \to (\tilde{x}/\tilde{M})_{\infty} )$$  is  {\em an isomorphism of the Cartan bundles} (resp. {\em a formal isomorphism of the Cartan bundles}) if $$\Phi^{(k)}: \sR^k(M) \to \sR^k(\tilde{M}) $$ $$ \big( \mbox{resp. }
 \varphi^{(k)}((\pi^k)^{-1}(x)/ \sR^{k}(M))_{\infty} \to ((\tilde{\pi}^k)^{-1}(\tilde{x})/
  \sR^{k}(\tilde{M}))_{\infty} \big) $$ sends   $P^k(M,B,G)$ to  $\tilde{P}^k(\tilde{M}, \tilde{B}, \tilde{G})$ and there exists a
biholomorphism (resp. formal isomorphism) $$h: B \to \tilde{B}
\ ( \mbox{resp. } f: (\pi_B(x)/B)_{\infty} \to (\pi_{\tilde{B}}(\tilde{x})/\tilde{B})_{\infty})$$ such that $\pi_{\tilde{B}} \circ \Phi = h \circ \pi_B$ (resp. $\pi_{\tilde{B}} \circ \varphi = f \circ \pi_B$).
 Suppose Cartan bundles $P^k:= P^k(M, B, G)$ on $M$ and   $\tilde{P}^k:= \tilde{P}^k(\tilde{M}, B, G)$ on $\tilde{M}$ have the same structural $B$-subgroup germ $ G \to B.$ Given  $z \in P^k$ and  $\tilde{z} \in \tilde{P}^k$  satisfying $$\pi_B \circ \pi^k (z) = \tilde{\pi}_B \circ \tilde{\pi}^k (\tilde{z}),$$
 we say that
 a biholomorphism $\Psi: (z/P^k)_{\sO} \to (\tilde{z}/\tilde{P}^k)_{\sO}$ is an {\em  isomorphism of Cartan bundles} if  $\Psi= \Phi^{(k)}$ for some biholomorphism $$\Phi: (x/M)_{\sO} \to (\tilde{x}/\tilde{M})_{\sO}, \ x= \pi^k(z), \tilde{x} = \tilde{\pi}^k(\tilde{z}).$$   Similarly,  a {\em  formal isomorphism of Cartan bundles}  $\psi: (z/P^k)_{\infty} \to (\tilde{z}/\tilde{P}^k)_{\infty}$ is of the form $\psi = \varphi^{(k)}$ for some formal isomorphism $\varphi: (x/M)_{\infty} \to (\tilde{x}/\tilde{M})_{\infty}$.

Restricting $\Theta^{k-1}$  to a Cartan bundle $P^k(M, B, G) \subset \sR^k(M)$, we obtain a $V^{k-1}$-valued form
$\theta^{k-1}$ and a $V^{k-2} \otimes \wedge^2V^*$-valued function $c^{k-2}$ on $P^k(M, B,G)$.
For a vector space $E$ and a manifold $Y$, let us denote by $E_Y$ the trivial vector bundle on $Y$ with the fiber $E$.
We say that a Cartan bundle $P^k= P^k(M, B, G^k)$ is {\em Morimoto-normal} if
$\theta^{k-1}$ has values in  the vector subbundle on $B$ $$ (V_B  \oplus  \fg^{k-1}) \subset (V + \fg^{k-1}(V))_B.$$ (As the word `normal' has different meaning in complex geometry, we use the term `Morimoto-normal' instead of `normal' used in \cite{Mo83}.) This is equivalent to saying that every point $z_k \in P^k$, regarded as a subspace in $T \sR^{k-1}(M)$, is tangent to $P^{k-1}(M, B, G^{k-1}) = \pi^k_{k-1}(P^k(M, B, G^k))$.
The restriction $c^{k-2}$ of   $C^{k-2}$ to $P^k$ is called the {\em first structure function } of $P^k$, which has values in $ ( V_B  \oplus \fg^{k-2} ) \otimes \wedge^2 V_B^*,$ if $P^k$ is Morimoto-normal.  The {\em second structure function} of $P^k$ is a holomorphic map $$\chi^k: P^k \to T B \otimes V_B^*$$  defined by $$\chi^k(z_k) (v) = {\rm d} \pi_B \left( \pi^k_0(z_k) (v) \right) \ \in \ T B \mbox{ for } z_k \in P^k, v \in V,$$ where $\pi^k_0(z_k) (v) \in T M$ is the image of $v$ under
 the isomorphism $ \pi^k_0(z_k): V \to T_{\pi^k(z_k)} M$.  It behaves equivariantly under the action of $G\to B$.

\subsection{}

We recall  the equivalence method for involutive Cartan bundles, from Section 8 of \cite{Mo83}.

Throughout, we fix a vector space $V$.  For a vector space $W$ and a subspace $\fh \subset \Hom( V, W)$, the first prolongation of $\fh$ is the subspace $\fh^{(1)} \subset \Hom (V, \fh)$ defined by
$$\fh^{(1)} := \{  h \in \Hom(V, \fh), \ h(u) v = h(v) u \mbox{ for all } u, v \in V\}.$$
For $v_1, \ldots, v_j \in V$, define $$\fh(v_1, \ldots, v_j) :=
\{ h \in \fh \subset \Hom(V, W), h(v_1) = \cdots = h(v_j) =0\}.$$
We say that  $\fh$ is {\em involutive} if there exists a basis $(v_1, \ldots, v_n)$ of $V$ such that $$ \dim \fh^{(1)} = \dim \fh + \sum_{i=1}^{n-1} \dim \fh(v_1, \ldots, v_i).$$
If $\fh$ is involutive, then $\fh^{(1)} \subset \Hom (V, \fh)$ is also involutive
(Proposition 8.2 in \cite{Mo83} due to Guillemin-Sternberg and Serre).

A Cartan bundle $P^0(M, B, G^0)$ of order 1 is {\em involutive} if
$$\fg^0(b) \subset \fg^0(V) = \Hom(V, V)$$ is involutive for any $b \in B$ and
the structure functions $c^{-2}$ and $\chi^{0}$ are $B$-constant, i.e.,
$$c^{-2}(z) = c^{-2}(z') \mbox{ and } \chi^{0}(z) = \chi^{0}(z')$$ for all $z, z' \in P^0 $ satisfying $ \pi_B(z) = \pi_B(z').$
For involutive Cartan bundles of order 1, we have the following result, which refines
Theorem 8.2 of \cite{Mo83} attributed to \cite{SS} and \cite{TU}.

\begin{theorem}\label{t.SS}
Let $P^0= P^0(M, B, G^0)$ and $\tilde{P}^0= \tilde{P}^0(\tilde{M},  B, G^0)$ be two involutive Cartan bundles of order 1, with the same structural $B$-group germ $G^0 \to B$.  Then for any positive integer $\ell$, any $(z,z') \in P^0 \times_B \tilde{P}^0$ and any formal isomorphism of Cartan bundles $\varphi: (z/ P^0)_{\infty} \to (\tilde{z}/\tilde{P}^0)_{\infty}$, there exists an isomorphism of Cartan bundles $\Phi: (z \in P^0)_{\sO} \to (\tilde{z} \in \tilde{P}^0)_{\sO}$   such that $$\Phi |_{(z/P^0)_{\ell}}
= \varphi|_{(z/P^0)_{\ell}}.$$ \end{theorem}

\begin{proof}
By Lemma 3.3 of \cite{SS} (or Lemma 5 of \cite{TU}), the involutiveness of $P^0$ and $\tilde{P}^0$ implies that the exterior differential system on the manifold $P^0 \times_B \tilde{P}^0$ characterizing local
biholomorphisms that are isomorphisms of Cartan bundles $P^0$ and $\tilde{P}^0$ is involutive in the sense of Cartan, hence in the sense  of Definition II.3.10 of \cite{Ma05} by Theorem 3.4 in Appendix B of \cite{Ma05}.  Thus given a formal solution of the differential system, we can find a local holomorphic solution approximating the formal solution up to any given order, by Cartan-K\"ahler theorem with estimate, e.g., Theorem 4.2 in Appendix of \cite{Ma72} (see also  Section III.3 in \cite{Ma05} or Theorem IX.2.2 of \cite{BCG}).  \end{proof}

\begin{remark}
T. Morimoto pointed out to me that one can also use Proposition 8.5 of \cite{Mo83} to lift the formal isomorphism $\varphi$ to higher order involutive Cartan bundles and apply the standard Cartan-K\"ahler theorem to deduce Theorem \ref{t.SS}. \end{remark}

Define $\fg_i(V) = {\rm Ker}(\varepsilon^i_{i-1}: \fg^i(V) \to \fg^{i-1}(V)).$ Then
there is a natural inclusion $\fg_{i}(V) \subset \Hom(V, \fg^{i-1}(V))$ (p. 302 of \cite{Mo83}).
For a regular Lie $B$-subgroup germ $G^k \to B$ of $G^k(V)$ and the associated Lie $B$-subalgebra $\fg^k \to B$ of $\fg^k(V)$,
we set $\fg^i = \varepsilon^k_i( \fg^k)$ and $\fg_i = \fg^i \cap \fg_i(V)$
with a natural inclusion $\fg_i \subset \Hom(V, \fg^{i-1}(V))$.
We write $\fg_i^{(1)} = \fg_{i+1}$ if they coincide under the inclusions
$$\fg_i^{(1)} \subset \Hom(V, \fg_i) \subset \Hom(V, \fg^i(V)) \mbox{ and }
   \fg_{i+1} \subset \Hom(V, \fg^i(V)).$$
A Cartan bundle $P^k(M, B, G^k)$ of order $k+1, k \geq 1$, is {\em involutive} if \begin{itemize}
\item[(i)] $P^k$ is Morimoto-normal;
\item[(ii)] the structure functions $c^{k-2}$ and $\chi^{k}$ are $B$-constant, i.e., $$c^{k-2}(z) = c^{k-2}(z') \mbox{ and } \chi^{k}(z) = \chi^{k}(z')$$ for all $ z, z' \in P^k $  satisfying $ \pi_B(z) = \pi_B(z').$ \item[(iii)]
    $\fg_{k-1}(b)$ is involutive for all $b \in B$; and
    \item[(iv)] $\fg_{k-1}(b)^{(1)} = \fg_k(b)$ for all $b \in B$.\end{itemize}
 Given a Cartan bundle $P^k(M, B, G^k)$ of order $k+1  \geq 2,$ if we  replace $B$ by a neighborhood of any point on $B$, we can view the projection $$P^k(M, B, G^k) \to P^{k-1}(M,B, G^{k-1})$$ as a Cartan bundle of order 1,  denoted by  $P^*(M, B, G_k^*)$, on the manifold
$P^{k-1}(M, B, G^{k-1})$ (Proposition 5.1 in \cite{Mo83}). If $P^k(M, B, G^k)$ is involutive, then $P^*(M, B, G_k^*)$ is involutive (Proposition 8.5 in  \cite{Mo83}).
Thus Theorem \ref{t.SS} implies the following refinement of Theorem 8.1 of \cite{Mo83}.

\begin{theorem}\label{t.GS}
Let $P^k= P^k(M, B, G^k)$ and $\tilde{P}^k = \tilde{P}^k(\tilde{M},  B, G^k)$ be two involutive Cartan bundles with the same structural $B$-group germ $G^k \to B$. Then for any positive integer $\ell$, any $(z,\tilde{z}) \in P^k \times_B \tilde{P}^k$ and any formal isomorphism of Cartan bundles $\varphi: (z/P^k)_{\infty} \to (\tilde{z}/\tilde{P}^k)_{\infty}$, there exists an isomorphism of Cartan bundles $\Phi: (z/P^k)_{\sO} \to (\tilde{z} /\tilde{P}^k)_{\sO}$ such that $$\Phi|_{(z/P^k)_{\ell}} = \varphi|_{(z/P^k)_{\ell}}.$$ \end{theorem}

\subsection{}

The following is a slightly modified version of Theorem 9.1  of \cite{Mo83}.

\begin{theorem}\label{t.Morimoto}
Let $P(M, B, G) \to M \stackrel{\pi_B}{\to} B$ be a Cartan bundle of order 1. Then there exists a  nowhere-dense subset $S \subset M$ with the following properties.
\begin{itemize}
\item[(i)] $S = \cup_{i=1}^r S_i$ for some positive integer $r$ where $S_1$ is a closed analytic subset of $M$ and $S_{j+1}$ is a closed analytic subset of $M\setminus \cup_{i=1}^{j} S_i$ for any $1 \leq j < r$.
    \item[(ii)]
For each $ x \in M\setminus S$, there exists a neighborhood $U \subset M \setminus S$ of $x$ and a positive integer $k_0$ such that for each integer $k \geq k_0$, one can construct by a finite algorithm,  in a unique manner compatible with isomorphisms up to conjugation, an involutive Cartan bundle $P^k(U, B', G^k)$ of order $k+1$ associated with a regular $B'$-subgroup germ $G^k \to B'$ over a complex manifold $B'$ with submersions $\pi_{B'}: U \to B'$ and $\pi^{B'}_B: B' \to B^x$  over an open neighborhood $B^x$ of $\pi_B(x)$ in $B$, satisfying $\pi_B |_U = \pi^{B'}_B \circ \pi_{B'}$. \item[(iii)] In (ii),  the formal isomorphism type of $P(M,B,G)$ at $x \in M$ determines whether $ x$ belongs to $S$ or not, and determines the formal structure of $P^k(U, B', G^k)$ at $x$ if $x \in M \setminus S.$ \end{itemize}   \end{theorem}

\begin{proof}
Excepting (iii), this is exactly (with slight changes of notation) Theorem 9.1 of \cite{Mo83}.
We sketch Morimoto's argument, explaining why it implies (iii).
The algorithm  in (ii), as explained in pages 349--351 of \cite{Mo83}, consists of three kinds of operations: projections by $\pi^{\ell}_0$, prolongations (Section 6 of \cite{Mo83}) and reductions (Section 7 of \cite{Mo83}).  These operations at each point $x \in M$ depend only on the formal structure of $P(M, B, G)$ at $x$.
Among them, the reduction step works only at points where the structure functions $c^{\ell-2}$ and $\chi^{\ell}$ of a Cartan bundle $P^{\ell}$ constructed inductively in  the algorithm are submersive over the smooth loci of their images. The closed analytic subset $S_i \subset M \setminus \cup_{j=1}^{i-1} S_j$ is the locus where the
structure functions appearing in the $i$-th step of the algorithm are not submersive over the smooth loci of their images (this is the condition given after Theorem 7.1 of \cite{Mo83}).  Thus the formal isomorphism type of $P(M, B, G)$ at $x$ determines whether $x$ belongs to $S$ or not, and also the formal structure of $P^k(U, B', G^k)$ at $x \in M \setminus S$.
\end{proof}

The combination of Theorem \ref{t.GS} and Theorem \ref{t.Morimoto} yields the following.

\begin{theorem}\label{t.Cartan}
Let $M, \tilde{M}$ be two complex manifolds with Cartan bundles $\pi_M: P= P(M, B, G) \to M$ and $\pi_{\tilde{M}}: \tilde{P} = \tilde{P}(\tilde{M}, B, G) \to \tilde{M}$ of the same order with  the same structural Lie $B$-group germs $G \to B$ with the associated holomorphic maps $$M \stackrel{ \pi_B}{\longrightarrow} B \stackrel{\tilde{ \pi}_B}{\longleftarrow} \tilde{M}.$$ Then there exists a nowhere-dense subset $S \subset M$ of the type described in (i) of Theorem \ref{t.Morimoto} such that \begin{itemize}
\item[(1)]  for any $z \in P$ and $\tilde{z} \in \tilde{P}$ satisfying $\pi_M(z) \not\in S$ and
$\pi_B \circ \pi_M (z) = \tilde{\pi}_{B} \circ \pi_{\tilde{M}} (\tilde{z})$ ;  \item[(2)] for any formal isomorphism of Cartan bundles $$\varphi: (z/P)_{\infty} \to (\tilde{z}/\tilde{P})_{\infty} $$
 satisfying $ \tilde{\pi}_{B} \circ \pi_{\tilde{M}} \circ \varphi = \pi_B \circ \pi_M|_{(z/P)_{\infty}};$ and \item[(3)] for any positive integer $\ell$, \end{itemize}  there exists an isomorphism of Cartan bundles
$\Phi: (z/P)_{\sO} \to (\tilde{z}/ \tilde{P})_{\sO}$ such that $\tilde{\pi}_{B} \circ \pi_{\tilde{M}} \circ \Phi = \pi_B \circ \pi_M|_{(z/P)_{\sO}}$ and $\Phi|_{(z/P)_{\ell}} = \varphi|_{(z/P)_{\ell}}$. \end{theorem}

    \section{Proofs of Theorems \ref{t.main} -- \ref{t.CF}}
For the proof of Theorem \ref{t.main}, it is convenient to introduce the following terminology.

    \begin{definition}\label{d.nice}
     A pair of holomorphic maps $\sK \stackrel{\rho}{\leftarrow} \sU \stackrel{\mu}{\rightarrow} \sX$ where $\sK, \sU, \sX$ are complex manifolds, is  a {\em nicely separating family} if the following properties hold.
      \begin{itemize} \item[(1)] $\rho$ is a proper submersion.
      \item[(2)] $\mu$ is a submersion  and embeds each fiber of $\rho$ into $\sX$ as a compact complex submanifold.  \item[(3)] When $p = \dim \sU - \dim \sX$, (2) implies that for each $u \in \sU$, the germ of $\rho(\mu^{-1}(\mu(u))$ at $\rho(u)$ is a $p$-dimensional submanifold   in $\sK.$ Define a holomorphic map $\rho': \sU \to {\rm Gr}(p, T \sK)$ to the Grassmannian bundle of the tangent bundle $T\sK$ by
    sending a point $u \in \sU$ to the tangent space of $\rho(\mu^{-1}(\mu(u))$ at $\rho(u)$.
Then $\rho'$ is injective. \end{itemize} \end{definition}

\begin{definition}\label{d.P0} In Definition \ref{d.nice}, consider the two vector subbundles $$T^{\rho} := {\rm Ker} ({\rm d} \rho), \ T^{\mu} := {\rm Ker}({\rm d} \mu) \ \subset T \sU.$$ They satisfy $T^{\rho} \cap T^{\mu} =0$. Fix a vector space $V$ with two subspaces $V_1, V_2 \subset V$ such that $$\dim V = \dim \sU, \ \dim V_1 = {\rm rank} T^{\rho}, \ \dim V_2 = {\rm rank} T^{\mu} \mbox{ and }  V_1 \cap V_2 =0.$$
Let  $P^0$ be the fiber subbundle of the frame bundle $\sR^0(\sU)$ whose fiber
 $P^0_y$ at $y \in \sU$ is  $$P^0_y := \{ h \in {\rm Isom}(V, T_y \sU)= \sR^0_y(\sU), \ h(V_1) = T^{\rho}_y \mbox{ and } h(V_2) = T^{\mu}_y \}.$$
We call $P^0$  the {\em canonical Cartan bundle associated with the nicely separating family} $\sK \stackrel{\rho}{\leftarrow} \sU \stackrel{\mu}{\rightarrow} \sX$.
Its structure group is the  subgroup of ${\rm GL}(V)$ preserving $V_1$ and $V_2$.
\end{definition}


    \begin{proposition}\label{p.CF}
    Let $\sK \stackrel{\rho}{\leftarrow} \sU \stackrel{\mu}{\rightarrow} \sX$ and
    $\tilde{\sK} \stackrel{\tilde{\rho}}{\leftarrow} \tilde{\sU} \stackrel{\tilde{\mu}}{\rightarrow} \tilde{\sX}$ be two nicely separating families, with the associated canonical Cartan bundles $P^0$ and $\tilde{P}^0$, respectively.  For $z \in \sU$ and $\tilde{z} \in \tilde{\sU}$,
    let $\Phi: O \to \tilde{O}$ be a biholomorphism between a neighborhood $O$ of $z$ and
    a neighborhood $\tilde{O}$ of $\tilde{z}$ which sends $z$ to $\tilde{z}$ and induces an isomorphism of the canonical Cartan bundles $P^0$ and $\tilde{P}^0$.  Then after shrinking $O$ and $\tilde{O}$ if necessary,   \begin{itemize}
    \item[(i)]
    $\Phi$ sends fibers of $\mu|_O$ to fibers of $\tilde{\mu}|_{\tilde{O}}$ descending to a biholomorphic map $\Phi^{\flat}: \mu(O) \to \tilde{\mu}(\tilde{O})$;
        \item[(ii)] $\Phi$ sends fibers of $\rho|_O$ to fibers of $\tilde{\rho}|_{\tilde{O}}$
        descending to a biholomorphic map $\Phi^{\sharp}: \rho(O) \to \tilde{\rho}(\tilde{O})$;
        \item[(iii)] there exists a biholomorphic map $F: U \to \tilde{U}$ from a neighborhood $U$ of $\rho^{-1}(\rho(z))$ to a neighborhood $\tilde{U}$ of $\tilde{\rho}^{-1}(\tilde{\rho}(\tilde{z}))$ such that the germ of $F$ at $z$ equals the germ of $\Phi$ at $z$; and \item[(iv)] the map $F$ induces a biholomorphic map $F^{\flat}$ from a neighborhood of $\mu(\rho^{-1}(\rho(z)))$ in $\sX$ to a neighborhood of $\tilde{\mu}(\tilde{\rho}^{-1}(\tilde{\rho}(\tilde{z})))$ in $\tilde{\sX}$ which agrees with $\Phi^{\flat}$ on $(\mu(z)/\sX)_{\sO}$.  \end{itemize} \end{proposition}

        \begin{proof}
        (i) and (ii) are immediate from the definition of the Cartan bundles in Definition \ref{d.P0}.
        By Definition \ref{d.nice}, we have injective holomorphic maps
        $$ \rho': \sU \to {\rm Gr}(p, T \sK) \mbox{ and } \tilde{\rho}': \tilde{\sU} \to {\rm Gr}(p, T \tilde{\sK}),$$ which can be regarded as the normalization of the subvarieties
        $$\rho'(\sU) \subset {\rm Gr}(p, T \sK) \mbox{ and } \tilde{\rho}'(\tilde{
        \sU}) \subset {\rm Gr}(p, T \tilde{\sK}).$$
        The biholomorphic map $\Phi^{\sharp}$ induces a biholomorphic map
        $${\rm d} \Phi^{\sharp}: {\rm Gr}(p, T \rho(O)) \to {\rm Gr}(p, T \tilde{\rho} (\tilde{O})).$$ By (i) and   (ii), it induces a biholomorphic map ${\rm d} \Phi^{\sharp}|_{\rho'(O)}: \rho'(O) \cong \tilde{\rho}'(\tilde{O}).$ Since $\rho'(O)$ (resp. $\tilde{\rho}'(\tilde{O})$) is an open subset of $\rho'(\sU)$ (resp. $\tilde{\rho}'(\tilde{\sU})$), it induces a biholomorphic map
        $$ {\rm d} \Phi^{\sharp}|_{\rho'(\sU)} : \rho'(\sU)\cap {\rm Gr}(p, T \rho(O)) \cong \tilde{\rho}'(\tilde{
        \sU}) \cap {\rm Gr}(p, T \tilde{\rho}(\tilde{O})).$$ This biholomorphism between two analytic subvarieties can be lifted naturally to a biholomorphic map $F$ of their normalizations:
$$        \begin{array}{ccc}
        \rho^{-1}(\rho(O)) =: U & \stackrel{F}{\longrightarrow}  & \tilde{U}:= \tilde{\rho}^{-1}(\tilde{\rho}(\tilde{O})) \\
        \rho' \downarrow & & \downarrow \tilde{\rho}' \\
        \rho'(\sU)\cap {\rm Gr}(p, T \rho(O)) & \stackrel{{\rm d} \Phi^{\sharp}|_{\rho'(\sU)}}{\longrightarrow}  & \tilde{\rho}'(\tilde{
        \sU}) \cap {\rm Gr}(p, T \tilde{\rho}(\tilde{O})). \end{array} $$
        It is clear that the germ of $F$ at $z$ equals the germ of $\Phi$ at $z$. Finally, (iv) is immediate from (i) and (iii). \end{proof}

       For the proof of Theorem \ref{t.main}, we need two lemmata. The following is a special case of Proposition 5.9 of \cite{Hi}.

       \begin{lemma}\label{l.Hi}
       Let $A \subset X$ be a compact complex submanifold in a complex manifold. Then for each nonnegative integer $\ell$, there exists  a positive integer $\ell^+$ such that if an  automorphism of the formal neighborhood $\psi_{\infty} : (A/X)_{\infty} \to (A/X)_{\infty}$ fixes  $(x/X)_{\ell^+}$ for a point $x \in A$, then the induced automorphism         
        $\psi_{\ell} : (A/X)_{\ell} \to (A/X)_{\ell}$ of the $\ell$-th order neighborhood of $A$ in $X$ must be the identity. \end{lemma}

The next lemma is proved in p. 509 of \cite{Hi} by applying
  Propositions 5.7 and 5.8 of  \cite{Hi}.

  \begin{lemma}\label{l.Hi2}
  Let $A \subset X$ and $\tilde{A} \subset \tilde{X}$ be compact complex submanifolds in complex manifolds
  with a formal isomorphism $\psi: (A/X)_{\infty} \to (\tilde{A}/\tilde{X})_{\infty}$. Let $a \in {\rm Douady}(X)$ (resp. $\tilde{a} \in {\rm Douady}(\tilde{X})$ ) be the  point corresponding to $A$ (resp. $\tilde{A}$).  Let $\sK \subset {\rm Douady}(X)$  (resp. $\tilde{\sK} \subset {\rm Douady}(\tilde{X})$) be a  neighborhood of $a$ (resp. $\tilde{a}$) and let $$\sK \stackrel{\rho}{\leftarrow} \sU \stackrel{\mu}{\to} X
\mbox{  (resp. } \tilde{\sK} \stackrel{\tilde{\rho}}{\leftarrow} \tilde{\sU} \stackrel{\tilde{\mu}}{\to} \tilde{X}) $$ be the universal family morphisms over $\sK$ (resp. $\tilde{\sK}$). Then $\psi$ induces formal isomorphisms $$\varphi: ( \rho^{-1}(a)/\sU)_{\infty} \to ( \tilde{\rho}^{-1}(\tilde{a})/\tilde{\sU})_{\infty}
\mbox{ and } \varphi^{\sharp}: (a/\sK)_{\infty} \to (\tilde{a}/\tilde{\sK})_{\infty},$$
which are compatible with the morphisms $\rho, \mu, \tilde{\rho}$ and $\tilde{\mu}$.\end{lemma}

       We prove a slightly refined version of Theorem \ref{t.main} as follows.

        \begin{theorem}\label{t.main2}
        Let $X$ be a complex manifold and let $\sK$ be a connected open subset in  the smooth loci of ${\rm Douady}(X)$ such that the associated universal family morphisms $\sK \stackrel{\rho}{\leftarrow} \sU \stackrel{\mu}{\rightarrow} \mu(\sU) \subset  X$ is a nicely separating family. Then there exists a nowhere-dense subset $\sS = \cup_{i=1}^r \sS_i \subset \sK$        such that \begin{itemize} \item[(1)] $\sS_1 \subset \sK$ is a  closed analytic subset; \item[(2)] $\sS_{j+1}$ is a closed analytic subset in the complex manifold $\sK \setminus \cup_{i=1}^{j} \sS_i$ for each $1 \leq j < r$; and \item[(3)] the submanifold of $X$ corresponding to a point of $\sK \setminus \sS$ satisfies the formal principle.\end{itemize} \end{theorem}

        \begin{proof}
        Apply Theorem \ref{t.Morimoto} to the canonical Cartan bundle $P^0$ on $M= \sU$ associated with the nicely separating family. We obtain the subset $S =\cup_{i=1}^r S_i \subset M$ satisfying the property in Theorem \ref{t.Morimoto}. For each $S_i$,
         let $S'_i \subset S_i$ be the subset consisting of irreducible components of the closed analytic subset $S_i$ in $M \setminus \cup_{j=1}^{i-1} S_j$ defined as follows: an irreducible component $R$ of $S_i$ is in $S'_i$ if and only if $\rho^{-1}(\rho(R)) = R$. Set $\sS_i := \rho(S'_i)$.
        As $\rho$ is a proper morphism and $S'_1$ is a closed analytic subset of $\sU$, (1) is immediate. To prove (2), assume that it holds for each $ 1 \leq j \leq k$ for some $k < r$.   If $a \in \sK \setminus \cup_{i=1}^k \sS_i$ is an accumulation point of $\sS'_{k+1}$, then each point of $\rho^{-1}(a)$ is an accumulation point of $S'_{k+1}$. Choose a point $u \in \rho^{-1}(a)$ outside $\cup_{i=1}^k S_i$.
        Then $u \in S_{k+1}$ by the closedness of $S_{k+1} $ in $\sU \setminus \cup_{i=1}^k S_i$. It follows that $a \in \sS'_{k+1}$. This shows that $\sS_{k+1}$ is closed in
        $\sK \setminus \cup_{i=1}^k \sS_i$. To show that $\sS_{k+1}$ is an analytic subset in $\sK \setminus \cup_{i=1}^k \sS_i$, it suffices to check it   in a neighborhood of each point $ a \in \sS_{k+1}$ in $\sK$. As $\rho$ is a submersion, we may  show that the intersection of $\rho^{-1} (\sS_{k+1})$ with a neighborhood of some point $u \in \rho^{-1}(a)$ in $\sU$ is analytic near $u$.
       This is obvious if we  choose $u$ outside $\cup_{i=1}^k S_i$.

       To prove (3), let $a \in \sK \setminus \sS$ and let $A \subset X$ be the corresponding submanifold. Let $\tilde{A} \subset \tilde{X}$ be a submanifold of a complex manifold with a formal isomorphism $\psi: (A/X)_{\infty} \to (\tilde{A}/\tilde{X})_{\infty}$.
     Let $\varphi$ and $\varphi^{\sharp}$ be the formal isomorphisms obtained by applying Lemma \ref{l.Hi2} to our $\psi$. As the smoothness of a point of a complex space is a formal property (e.g. by Corollary 1.6 of \cite{At}), we see that $\tilde{a}$ is a smooth point of  $\tilde{\sK}$ and $\tilde{A}$ is unobstructed in $\tilde{X}$. Moreover, the normal bundle $N_{\tilde{A}/\tilde{X}}$ is isomorphic to $N_{A/X}$. Thus we can choose a neighborhood $\tilde{\sK} \subset {\rm Douady}(\tilde{X})$ of the point $\tilde{a} \in {\rm Douady}(\tilde{X})$  giving a nicely separating family $$ \tilde{\sK} \stackrel{\tilde{\rho}}{\leftarrow} \tilde{\sU} \stackrel{\tilde{\mu}}{\rightarrow} \tilde{\mu}(\tilde{\sU}) \subset \tilde{X}.$$
       We have the canonical Cartan bundles $P^0$ on $\sU$ and $\tilde{P}^0$ on $\tilde{\sU}$.
  As $\varphi$ is compatible with $\rho, \tilde{\rho},
\mu$ and $\tilde{\mu}$, its restriction to any $u \in \rho^{-1}(a)$ and $\tilde{u} = \varphi(u),$  $$\varphi_u: (u/\sU)_{\infty} \to (\tilde{u}/\tilde{\sU})_{\infty}$$   gives a formal isomorphism of the Cartan bundles $P^0$ and $\tilde{P}^0$.

       For a given positive integer $\ell$, let $\ell^+$ be as in Lemma \ref{l.Hi}.
              By our definition of $\sS$, there exists a point $u \in \rho^{-1}(a)$ which is not contained in $S$.  Applying Theorem \ref{t.Cartan},  we have a biholomorphic map  $$\Phi: (u/\sU)_{\sO}  \to (\tilde{u}/\tilde{\sU})_{\sO} \mbox{ with } \Phi|_{(u/\sU)_{\ell^+}} = \varphi|_{(u/
       \sU)_{\ell^+}}.$$
       Then Proposition \ref{p.CF} gives a biholomorphic map $$F^{\flat}:
       (A/X)_{\sO} \to (\tilde{A}/\tilde{X})_{\sO} \mbox{ with } F^{\flat}|_{(\mu(u)/X)_{\ell^+}} = \psi|_{(\mu(u)/X)_{\ell^+}}.$$  Thus
       the composition $\psi^{-1} \circ F^{\flat}|_{(A/X)_{\infty}}$ defines an automorphism of $(A/X)_{\infty}$ that fixes $(\mu(u)/ X)_{\ell^+}$. Thus by Lemma \ref{l.Hi}, the induced automorphism of $(A/X)_{\ell}$ is the identity, which means $F^{\flat}|_{(A/X)_{\ell}} = \psi|_{(A/X)_{\ell}}.$ Thus $A \subset X$ satisfies the formal principle, which proves (3).
       \end{proof}

Let us derive Theorem \ref{t.bundle} from Theorem \ref{t.main}.

\begin{proof}[Proof of Theorem \ref{t.bundle}]
 In the setting of Theorem \ref{t.bundle},  let $X$ be the underlying complex manifold of the bundle $W$ and let $0_A \subset X$ be the submanifold corresponding to the zero section of $W$. For a section $s \in H^0(A, W)$, denote by $s(A) \subset X$ the submanifold given by the values of the section $s$.

 We claim that the formal principle holds for $0_A \subset X,$ if it holds for $s(A) \subset X$ for some $s \in H^0(A, W)$.  In fact, the translation $\tau_s: X \to X$ which sends $w \in W_a$ on the fiber $W_a, a \in A$ to $\tau_s(w) = w + s(a)$ is a biholomorphic automorphism of $X$ that sends
$(0_A/X)_{\sO}$ to $(s(A)/X)_{\sO}$.

Consider the collection of submanifolds $\{s(A) \subset X, s \in H^0(A, W)\}$ and let  $\sK \subset {\rm Douady}(X)$ be the corresponding subset.  As each $s(A)$ is unobstructed, the set  $\sK$ is in the smooth loci of ${\rm Douady}(X)$. From $$\dim \sK = \dim H^0(A, W) =  \dim H^0(s(A), N_{s(A)/X}),$$  we see that $\sK$ is a connected open subset in ${\rm Douady}(X)$. Thus we are in the setting of Theorem \ref{t.main} with  the conditions (i) and (ii) satisfied. By the assumption that $W$ is globally generated, the condition (iii) of Theorem \ref{t.main} is satisfied. The condition (iv) is precisely the additional assumption of Theorem \ref{t.bundle}. Thus by Theorem \ref{t.main}, the formal principle holds for
$s(A) \subset X$ for some $s \in H^0(A, W)$, hence  for $0_A \subset X$ by the above claim.
 \end{proof}

Now we turn to Theorem \ref{t.Fano}. The following lemma is immediate from the ampleness of $K_A^{-1}$ of a Fano manifold $A$.

\begin{lemma}\label{l.Fano}
Let $A$ be a Fano manifold. Then there exists a positive integer $p$ such that
the $p$-th tensor power $K^{-p}_A := (K_A^{-1})^{\otimes p}$ satisfies
$$H^0(A, K_A^{-p} \otimes {\bf m}_x) \neq H^0(A, K_A^{-p} \otimes {\bf m}_y) \mbox{ for any two points } x \neq y \in A.$$ \end{lemma}

\begin{proposition}\label{p.L}
Let $A$ be a Fano manifold and let $p$ be a positive integer from Lemma \ref{l.Fano}. Let $A \subset X$ be an embedding in a complex manifold. Let $L$ be the underlying complex manifold of the line bundle $K^{-p}_X$ on $X$. Regard $X$ as the submanifold of $L$ defined by the zero-section of  the line bundle $K_X^{-p}$ and regard $A$ as a submanifold of $L$ via $A \subset X \subset L$. If the normal bundle $N_{A/X}$ is globally generated, then so is $N_{A/L}$ and for any $x \neq y \in A$, we have $$H^0(A, N_{A/L} \otimes {\bf m}_x) \neq H^0(A, N_{A/L} \otimes {\bf m}_y)$$ as subspaces in $H^0(A, N_{A/L})$. \end{proposition}

\begin{proof}
We claim that
for any $x \neq y \in A$,
$$H^0(A, K^{-p}_X|_A \otimes {\bf m}_x ) \neq H^0(A, K^{-p}_X|_A \otimes {\bf m}_y).$$
As $N_{A/X}$ is globally generated, there exists a section $\eta$ of $(\det N_{A/X})^{\otimes p}$  satisfying $\eta(x) \neq 0 \neq \eta(y)$. From $K^{-p}_X|_A = K^{-p}_A \otimes (\det N_{A/X})^{\otimes p}$, we have
 \begin{eqnarray*} H^0(A, K^{-p}_A \otimes {\bf m}_x) \otimes \eta & \subset & H^0(A, K^{-p}_X|_A \otimes {\bf m}_x), \\  H^0(A, K^{-p}_A \otimes {\bf m}_y) \otimes \eta & \subset & H^0(A, K^{-p}_X|_A \otimes {\bf m}_y). \end{eqnarray*}  Since $H^0(A, K^{-p}_A \otimes {\bf m}_x) \otimes \eta \neq H^0(A, K^{-p}_A \otimes {\bf m}_y) \otimes \eta$ by Lemma \ref{l.Fano},  we obtain the claim.

We have a direct sum decomposition of vector bundles on $A$:
$$ N_{A/L} \cong N_{A/X} \oplus  K^{-p}_X|_A.$$
Thus $N_{A/L}$ is globally generated.
 Moreover, for $x \neq y \in A$,  \begin{eqnarray*} H^0(A, N_{A/L} \otimes {\bf m}_x)
&=& H^0(A, N_{A/X} \otimes {\bf m}_x) \oplus  H^0(A, K^{-p}_X|_A \otimes {\bf m}_x) \\
  H^0(A, N_{A/L} \otimes {\bf m}_y) &=& H^0(A, N_{A/X} \otimes {\bf m}_y) \oplus H^0(A, K^{-p}_X|_A \otimes {\bf m}_y).\end{eqnarray*} Thus $H^0(A, N_{A/L} \otimes {\bf m}_x) \neq H^0(A, N_{A/L} \otimes {\bf m}_y)$ follows from the above claim.
\end{proof}

\begin{proof}[Proof of Theorem \ref{t.Fano}]
Let $W$ be a globally generated vector bundle on a Fano manifold $A$.
Let $X$ be the underlying complex manifold of $W$ and regard $A$ as a submanifold of $X$ by the zero section of $W$.
Let $\beta:
X \to A$ be the natural projection of the vector bundle $W$.

We claim that the line bundle $K^{-p}_X \to X$ is isomorphic to the pull-back $\beta^* (K_A^{-p} \otimes (\det W)^{\otimes p}).$ Note that the relative tangent bundle $$T^{\beta} := {\rm Ker}({\rm d} \beta : TX \to \beta^* T A)$$ can be canonically identified with  $\beta^* W$. Thus the  exact sequence
$$0 \to T^{\beta} \to T X \to \beta^* TA \to 0$$
 gives $$K^{-1}_X = \det T^{\beta} \otimes \beta^* K^{-1}_A = \beta^* \det W \otimes \beta^* K^{-1}_A,$$  which proves the claim.

 Let $L $ be the complex manifold constructed from $A \subset X$ in Proposition \ref{p.L} and let $\lambda: L \to X$ be the natural projection.   By the above claim, we can regard the composition $L \stackrel{\lambda}{\to}  X \stackrel{\beta}{\to} A$ as the underlying complex manifold of the direct sum of vector  bundles on $A$ $$W':= W \oplus (K^{-p}_A \otimes (\det W)^{\otimes p}).$$
 Proposition \ref{p.L} implies that the vector bundle $W'$ satisfies the condition of Theorem \ref{t.bundle}. We conclude that the formal principle holds for the inclusion $A \subset L.$

To prove that $A\subset X$ satisfies the formal principle,
let $\tilde{A} \subset \tilde{X}$ be an embedding of $A$ in another complex manifold $\tilde{X}$ with a formal isomorphism $\psi: (A/X)_{\infty} \to (\tilde{A}/\tilde{X})_{\infty}$ and pick a positive integer $\ell$.   Let $\tilde{L}$ be the underlying complex manifold of the line bundle $K^{-p}_{\tilde{X}}$ on $\tilde{X}$ with the natural projection  $\tilde{\lambda}: \tilde{L} \to \tilde{X}$.
The formal isomorphism $\psi$ gives rise to a formal isomorphism of the formal neighborhoods of the zero sections in the line bundles
$${\rm d} \psi: (0_A/\det TX)_{\infty} \to (0_{\tilde{A}}/ \det T \tilde{X})_{\infty},$$ which induces a formal isomorphism $$\varphi: (A /L)_{\infty} \to (\tilde{A}/\tilde{L})_{\infty} \mbox{ satisfying }
\varphi|_{(A/X)_{\infty}} = \psi.$$ Since $A \subset L$ satisfies the formal principle,  there exists a biholomorphic map $\Phi: (A/L)_{\sO} \to (\tilde{A}/\tilde{L})_{\sO}$  such that $$\Phi|_{(A/L)_{\ell}}
= \varphi|_{(A/L)_{\ell}}.$$
Write $\tilde{X}':= \Phi ((A/L)_{\sO} \cap X) = \Phi((A/X)_{\sO})$. Then $$(\tilde{A}/\tilde{X})_{\ell} \subset \tilde{X}' \cap \tilde{X}.$$ In particular, the submanifold $\tilde{X'} \subset \tilde{L}$ is tangent to the submanifold $\tilde{X} \subset \tilde{L}$ along $\tilde{A}$. Thus the projection
$\tilde{\lambda}: \tilde{L} \to \tilde{X}$ induces a biholomorphism
 $$\tilde{\lambda}|_{\tilde{X}'}: (\tilde{A}/\tilde{X}')_{\sO} \to (\tilde{A}/\tilde{X})_{\sO}.$$
 Then the biholomorphic map $$ \Psi:=\tilde{\lambda}|_{\tilde{X}'} \circ \Phi|_{(A/X)_{\sO}} \ : \
 (A/X)_{\sO} \to (\tilde{A}/\tilde{X})_{\sO}$$ satisfies $$\Psi|_{(A/X)_{\ell}} = \Phi|_{(A/X)_{\ell}} = \varphi |_{(A/X)_{\ell}} =  \psi |_{(A/X)_{\ell}}.$$ This proves that $A \subset X$ satisfies the formal principle. \end{proof}

\begin{proof}[Proof of Theorem \ref{t.free}]
The normal bundle of the smooth rational curve $A \subset X$ is of the form
$$N_{A/X} \cong \sO(m_1) \oplus \sO(m_2) \oplus \cdots \oplus \sO(m_r)$$
for some integers $m_1 \geq m_2 \geq \cdots \geq m_r$ and $ r= \dim X -1$.
As $N_{A/X}$ is globally generated, we have $m_r \geq 0$.

Suppose that $m_1= \cdots = m_r=0$, i.e., the normal bundle is trivial.
Then the morphism $\mu$ in Theorem \ref{t.main} is biholomorphic
over a neighborhood of $A$ and the formal principle holds for $A \subset X$ trivially.

Suppose that $m_1>0$. Then for any $x\neq y \in A$, we have
$$H^0(A, \sO(m_1) \otimes {\bf m}_x) \neq H^0(A, \sO(m_1) \otimes {\bf m}_y),$$
which implies  the condition (iv) of Theorem \ref{t.main}.
Thus we can apply Theorem \ref{t.main} to finish the proof. \end{proof}

Let us recall the following  version of
Cartan-Fubini type extension theorem.

\begin{theorem}\label{t.CFO}
Let $X, \tilde{X}$ be Fano manifolds of Picard number 1. Let $\sK$ (resp. $ \tilde{\sK}$) be an irreducible component of the space of rational curves on $X$ (resp.  $\tilde{X}$)  such that the subscheme
$\sK_x \subset \sK$ (resp. $\tilde{\sK}_{\tilde{x}}$) consisting of members through a general point
$x \in X$ (resp. $\tilde{x} \in \tilde{X}$) is nonempty, projective and irreducible.  Then there exits a nowhere-dense algebraic subset $\sK' \subset \sK$ such that  for any member  $A \subset X$ belonging to  $\sK \setminus \sK'$, if there exists a member $\tilde{A}$ of $\tilde{\sK}$ equipped with a biholomorphic map
 $$\Phi: (\Gamma_A/ (\BP^1 \times X))_{\sO} \to (\Gamma_{\tilde{A}}/ (\BP^1 \times \tilde{X}))_{\sO}$$ where
 $\Gamma_A \subset \BP^1 \times X$ (resp. $\Gamma_{\tilde{A}} \subset \BP^1 \times \tilde{X}$) is the graph  of the normalization of $A$ (resp. $\tilde{A}$),  then $\Phi$ can be extended to a biholomorphic map from $X$ to $\tilde{X}$. \end{theorem}

\begin{proof}[Sketch of proof of Theorem \ref{t.CFO}]
The assumption of Theorem \ref{t.CFO} implies the conclusion of Proposition 2.1 of \cite{HM01},
while the conclusion of Theorem \ref{t.CFO} is the conclusion of Main Theorem in p. 564 of
\cite{HM01}. Thus the proof of Theorem \ref{t.CFO} is contained in
the derivation of Main Theorem or Theorem 1.2 of
\cite{HM01}
 from Proposition 2.1 of \cite{HM01}. This derivation is exactly Section 3 and Section 4  of \cite{HM01}.
 \end{proof}

\begin{proof}[Proof of Theorem \ref{t.CF}]  Note that general members of any family of smooth rational curves on a complex manifold $X$ whose loci contain an open subset of $X$ have globally generated normal bundles (e.g. Proof of Theorem II.3.11 in \cite{Ko}). Thus in the setting of Theorem \ref{t.CF}, we can apply Theorem  \ref{t.free} to see that there exists a biholomorphic map
  $$\Phi: (\Gamma_A/ (\BP^1 \times X))_{\sO} \to (\Gamma_{\tilde{A}}/ (\BP^1 \times \tilde{X}))_{\sO}.$$ Then we are in the setting of  Theorem \ref{t.CFO}, which gives the extension to a biholomorphic map from $X$ to $\tilde{X}$. \end{proof}

We remark that Cartan-Fubini type extension theorems, like
 Main Theorem or Theorem 1.2 of \cite{HM01}, are usually formulated in terms of varieties of minimal rational tangents. But once one reaches the setting of Theorem \ref{t.CFO}, or the conclusion of Proposition 2.1 of \cite{HM01},  varieties of
 minimal rational tangents play no more role.


\begin{thebibliography}{KSWZ}

\bibitem{ABT} Abate, M., Bracci, F. and Tovena, F.: Embeddings of submanifolds and normal bundles, Adv. Math. {\bf 220} (2009) 620--656

\bibitem{Ar} Arnol'd, V. I.: Bifurcations of invariant manifolds of differential equations and normal forms in neighborhoods of elliptic curves. Funct. Anal. and Appl. {\bf 10} (1976) 249--259

\bibitem{At} Artin, M.: On the solutions of analytic equations. Invent. Math. {\bf 5} (1968) 277-291

\bibitem{BCG} Bryant, R., Chern, S.-s., Gardner, R. B., Goldschmidt, H. and Griffiths, P.: {\em Exterior differential systems}. MSRI Publications, Springer, 1990

\bibitem{Ca} Cartan, \'E.: Les sous-groupes des groupes coninus de transformations. Ann. \'Ecole Norm. Supp. {\bf 25} (1908) 57--194

\bibitem{CG} Commichau, M. and Grauert, H.: Das formale Prinzip f\"ur kompakte komplexe Untermannigfaltigkeiten mit 1-positivem Normalenb\"undel.
    {\em Recent developments in several complex variables}, Ann. Math. Stud. {\bf 100} (1981) 101--126

\bibitem{Dou} Douady, A.: Le probl\`eme des modules pour les sous-espaces analyticques compacts d'un espace anlaytique donn\'e. Ann. Inst. Fourier {\bf 16} (1966) 1--95


\bibitem{GP}  Grauert, H., Peternell, T. and Remmert, R.: {\em Several complex variables VII}. Encycl. Math. Sci. {\bf 74},  Springer-Verlag, Berlin-Heidelberg, 1994


\bibitem{Gi66} Griffiths, P.: The extension problem in complex analysis II; Embeddings with positive normal bundle. Amer. J. Math.
{\bf 88} (1966) 366--446

\bibitem{Hi}
Hirschowitz, A.:
On the convergence of formal equivalence between embeddings.
Ann. of Math. {\bf 113} (1981) 501--514

\bibitem{HM01}  Hwang, J.-M. and  Mok, N.: Cartan-Fubini type extension of holomorphic maps for Fano manfiolds of Picard number 1. J. Math. Pures Appl. (9) {\bf 80} (2001) 563--575





\bibitem{Ko}  Koll\'ar, J.: {\em Rational curves on algebraic varieties}. Springer-Verlag, Berlin-Heidelberg-New York, 1996

\bibitem{Ks86} Kosarew, S.: On some new results on the formal principle for embeddings. {\em Proceedings of the conference on algebraic geometry (Berlin, 1985)}, 217--227, Teubner-Texte Math. {\bf  92}, Teubner, Leipzig, 1986

\bibitem{Ks88} Kosarew, S.:
Ein allgemeines Kriterium fur das formale Prinzip.  J. Reine Angew. Math. 388 (1988) 18--39

\bibitem{Ma72} Malgrange, B.: Equation de Lie II, J. Diff. Geom. {\bf 7} (1972) 117--141

\bibitem{Ma05} Malgrange, B.: {\em Syst\`emes diff\'erentiels involutifs}. Panoramas et Synth\`ese {\bf 19}. Socie\'et\'e Mathe\'ematique de France, 2005



 \bibitem{Mo81} Morimoto, T.: Sur le probleme d'\'equivalence des structures g\'eom\'etriques.  C. R. Acad. Sci. Paris Ser. I Math. {\bf 292} (1981)  63--66

     \bibitem{Mo83}
    Morimoto, T.:  Sur le probl\`eme d'\'equivalence des structures g\'eom\'etriques. Japan. J. Math. {\bf 9} (1983) 293--372

    \bibitem{SS} Singer, I. M. and Sternberg, S.:
    The infinite groups of Lie and Cartan, Part I, (The transitive
groups). J. d'Analyse Math. {\bf 15} (1965) 1--114



\bibitem{St} Steinbiss, V.:
Das formale Prinzip fur reduzierte komplexe Raume mit einer schwachen Positivitatseigenschaft.
Math. Ann. {\bf 274} (1986) 485--502

\bibitem{TU} Ueno, K.: A study on the equivalence of generalized G-structures.
M. Sc. thesis, Kyoto University, 1968

\end{thebibliography}
\end{document}